\documentclass[12pt,reqno]{amsart}

\usepackage{graphicx,subfigure}
\usepackage{hyperref}
\hypersetup{colorlinks=true, citecolor=blue, linkcolor=red}

\numberwithin{equation}{section} \numberwithin{figure}{section}
\numberwithin{table}{section} 
{ \theoremstyle{plain}

}

{ \theoremstyle{definition}
\newtheorem{thm}{Theorem}

\newtheorem{rem}{Remark}

\numberwithin{equation}{section} \numberwithin{lem}{section}
\numberwithin{thm}{section} \numberwithin{cor}{section}
\numberwithin{pro}{section} \numberwithin{rem}{section}
}

\begin{document}

\title{The heteroclinic connection problem for general double-well potentials}


\author{Christos Sourdis}
\address{Department of Mathematics,  University of
Turin,   Via Carlo Alberto 10,
20123, Turin, Italy.}
\email{christos.sourdis@unito.it}

\keywords{heteroclinic connection, variational methods, Hamiltonian systems, phase transitions}

\footnote{\textup{2000} \textit{Mathematics Subject Classification}:  65M60, 65M50, 35Q55}




\maketitle

\begin{abstract}
By variational methods, we provide a simple proof of  existence of a
heteroclinic orbit to the Hamiltonian system $u''=\nabla W(u)$ that
connects the two global minima of a double-well potential $W$.
Moreover, we consider several inhomogeneous
extensions.
\end{abstract}

\section{Introduction}
\subsection{The problem}
In this paper, we will prove existence of solutions $u\in
C^2(\mathbb{R},\mathbb{R}^n)$ to the following problem:
\begin{equation}\label{eqEq}
u_{xx}=\nabla W(u),\ \ x\in \mathbb{R},\ \ \lim_{x\to \pm
\infty}u(x)=a_\pm,
\end{equation}
where \begin{equation}\label{eqass1}W\in C^1(\mathbb{R}^n),\ n\geq
1,\ \textrm{satisfies}\ W(a_-)=W(a_+)=0,\  W(u)>0\ \textrm{if}\ u\neq
a_\pm,\end{equation} for some $a_-\neq a_+$, and
the function
\begin{equation}\label{eqomega}
\omega(s)=\min_{|u|=s}W(u),\ \ s\geq 0,
\end{equation}
satisfies
\begin{equation}\label{eqLeoni}
\int_{0}^\infty \sqrt{\omega(s)}ds=\infty.
\end{equation}
In passing, we note that $\omega$ is upper semicontinuous. In fact, we will prove a more general result (see Theorem \ref{thmThm} below).

Since  $a_-\neq a_+$, such a solution is called a \emph{heteroclinic
connection}, as opposed to a homoclinic. Motivated from mechanics,
in relation with Newton's second law of motion (where $x$ plays the
role of time), we will often refer to $W$ as a double-well potential
(see also \cite{alikakosIndiana,bonhoure} and the references therein). In order to avoid confusion, we point out that $-W$ is what is usually referred to as the potential in classical mechanics.

We note that the quantity
\[
\frac{1}{2}|u_x|^2-W(u)
\]
is constant along solutions of the equation, which easily  implies
that $W(a_-)=W(a_+)$ is a necessary condition for a heteroclinic
connection to exist between $a_-$ and $a_+$.

We will also study the inhomogeneous problem
\begin{equation}\label{eqinhom}
u_{xx}=h(x)\nabla W(u),\ \ \lim_{x\to \pm \infty} u(x)=a_\pm,
\end{equation}
under various assumptions on $h$.

\subsection{Motivation}\label{subsecMotiv}
The theory of phase transitions  has led to the extensive study of
singularly perturbed, non-convex energies of the form
\[
J_\varepsilon(u)=\int_{\Omega}^{}\left\{\frac{\varepsilon}{2}|\nabla
u|^2+\frac{1}{\varepsilon}W(u) \right\}dx,
\]
where $W$ is a nonnegative potential with a finite number of global minima; usually these are assumed to be nondegenerate, and that $W$ is coercive at infinity or at least that  \begin{equation}\label{eqinf}\liminf_{|u|\to \infty}
W(u)>0.\end{equation}
 In
the scalar case, this problem was studied by Modica \cite{modica}
using De Giorgi's notion of $\Gamma$-convergence (see also
\cite{alberti,braides} and the references therein). In the vectorial case of
two global minima, that is when (\ref{eqass1}) and (\ref{eqinf}) hold,
the $\Gamma$-limit of this energy was studied in
\cite{barroscoFosec}, \cite{fonsecaTartarEdin} (for a thorough discussion around condition (\ref{eqinf}) in this context, we refer to \cite{leonESaiM}). The case where $W$
has more than two wells was considered in \cite{baldo} (see also
\cite{sternberg}). In this context, the heteroclinic connections
determine the interfacial energy.

In parallel,  the interest in the heteroclinic connection problem
stems also from the study of the vectorial Allen-Cahn equation that
models multi-phase transitions (see \cite{alamaGui}, \cite{alberti},
\cite{alesio}, \cite{alikakosIndiana}, \cite{bethuel}, \cite{bgs},
and the references therein). Loosely speaking, the heteroclinic
connections are expected to describe the way in which the solutions
to the multi-dimensional parabolic system
\[
u_t=\varepsilon^2 \Delta u-\nabla W(u),
\]
for small $\varepsilon>0$, transition from one state to the other
(see \cite{BR}).

 The heteroclinic connection problem also comes up when
studying phase coexistence in consolidating porous medium (see
\cite{italians} and the references therein),  crystalline grain
boundaries (see \cite{braun}), planar transition front solutions to
the Cahn-Hilliard system \cite{howard},  and  domain walls in
coupled Gross-Pitaevskii equations (see \cite{alamaPeli,goldman} and the
references therein).

We emphasize that  some of these applications  require a triple-well or four-well
potential. Nevertheless, under a reflection symmetry assumption on
$W$ (which is frequently inherited from the physical model), the
problem can easily  be reduced to the double-well case (see
\cite{alamaPeli} or \cite{stefanos}). The heteroclinic connection problem for nonsymmetric  multi-well potentials is much more complicated to treat (see \cite{zunuga}).

For an application which requires one to  consider potentials with
degenerate minima, we refer to \cite{ball}.

Our motivation for the  inhomogeneous problems is twofold:

In \cite{nakashima}, among other things, by employing singular
perturbation techniques, the author constructed heteroclinic
connections to the scalar spatially inhomogeneous Allen-Cahn
equation
\begin{equation}\label{eqNaka}
u_{xx}=h(\varepsilon x) W'(u)\ \ \textrm{such that} \ \ \lim_{x\to
\pm \infty}u(x)=a_\pm,
\end{equation}
provided that $\varepsilon>0$ is sufficiently small, where $W$ has
the same features as in the present paper but assuming
non-degeneracy of the global minima; $h$ is strictly positive,
bounded, and having at least one non-degenerate local minimum. The
result relies on the fact that the $\varepsilon=0$ limit problem has
a unique, asymptotically stable heteroclinic solution. Our results
provide existence for all $\varepsilon>0$ and hold for systems with
more general $W$. Moreover, we believe that, with some more effort,
they can provide information about the $\varepsilon\to 0$ asymptotic
behavior of the solutions.

Recently, there has been an interest in constructing heteroclinic
solutions to  semilinear elliptic systems with variational structure (see
\cite{alikakosPreprint}). In that case, in order to exclude the
possibility of constructing the one dimensional heteroclinic, one
has to impose some spatial inhomogeneity to the problem.
We believe that  our approach, a
 refinement of that of \cite{alikakosIndiana,alikakosPreprint}, has the
advantage of being flexible enough to potentially treat the case of
these semilinear elliptic systems.

\subsection{Known results}
The problem (\ref{eqEq}) is completely understood if $n=1$, see for
instance \cite{alberti}, \cite{bonhoure}; in fact, assumption
(\ref{eqinf}) is not needed in that case.

If $n\geq 2$, under assumptions (\ref{eqass1}) and (\ref{eqinf}), the
existence of a heteroclinic orbit was proven  in
\cite{rabinowitz} via a variational approach (see also \cite[Thm.
2.3]{bonhoure}).

Under various additional nondegeneracy or geometric conditions near
the global minima of $W$, this problem has been dealt, mostly as a
tangential issue, in several references. Under the assumption that
\begin{equation}\label{eqAF}W(a_\pm+\rho \nu)\ \  \textrm{is\ increasing\ in}\ \rho \in [0,\delta],\ \forall \  \nu \in \mathbb{S}^{n-1},\end{equation} (for
some small $\delta>0$), where $\mathbb{S}^{n-1}$ stands for the unit
sphere, the existence of a heteroclinic connection was proven
recently  in \cite{alikakosIndiana} (see also \cite{alikKATZ} and \cite{stefanos}). Their
novelty was to employ constraints which are subsequently removed. It is worthwhile noting that they observed that their
proof goes through even when (\ref{eqinf}) is replaced with the  assumption that the function $\omega$ in (\ref{eqomega}) is decreasing as $s\to \infty$ and
\[
s^2\omega (s)\to \infty \ \ \textrm{as}\ \ s\to \infty,
\]
(compare with (\ref{eqLeoni}), which was also found independently in \cite{monteil}).
 If
$W(a_\pm+\rho \nu)\geq c\rho^\gamma$, $\rho \in [0,\delta]$, for
some $c,\gamma,\delta>0$, and assuming that the level sets of $W$
near $a_\pm$ are strictly convex, the existence of a heteroclinic
connection was proven very recently  in
\cite{katzourakis} in the spirit of the concentrated compactness
method. If the global minima of $W$ are non-degenerate, that is the
Hessian $\partial^2W(a_\pm)$ is positive definite, the existence of
a heteroclinic connection was proven  in
\cite{sternberg} by using techniques from $\Gamma$-convergence
theory  (an additional growth condition as $|u|\to \infty$ was also
assumed).
 Other variational proofs, which usually require some non-degeneracy of
the global minima, can be found in \cite{alamaGui}, \cite{alamaPeli},
\cite{alesio}, \cite{bgs}, \cite{goldman} and \cite{schatzman}. In fact, as is
pointed out, the proof of \cite{alamaPeli} carries over to the case
where  $W$ vanishes to finite order at $a_\pm$.

 To the best of our knowledge, there are only a few    corresponding results
 for spatially inhomogeneous systems, which we will refer to in the subsequent remarks. For the state of the art in the case of the scalar problem, we refer the interested reader to \cite{alves} and the references therein.

Lastly, we note that homoclinic and periodic orbits for  conservative systems as in (\ref{eqEq}) can also be studied variationally (see for instance \cite{cardioli} and \cite{weinstein} respectively); of course there are corresponding spatially inhomogeneous extensions.

\subsection{The main result} In view of the above discussion, and motivated by related literature
(see Remarks \ref{rem1}, \ref{rem2}, \ref{rem3}, \ref{rem4}  below), it is natural to embed problem
(\ref{eqEq}) into the more general family of inhomogeneous problems (\ref{eqinhom}) with $h$ positive and periodic.

Our primary goal is to prove the following
theorem.

\begin{thm}\label{thmThm}
\emph{Let
	$h\in C(\mathbb{R};\mathbb{R})$ be $T$-periodic ($T>0$) and satisfy
	\begin{equation}
	\label{eqh0G}
	h(x)\geq h_0>0,\ \ x\in \mathbb{R},
	\end{equation}
	for some constant $h_0$.
	Then, under assumptions (\ref{eqass1}) and (\ref{eqLeoni}), there exists a
solution $u\in C^2(\mathbb{R},\mathbb{R}^n)$ to the problem
(\ref{eqinhom}).
}\end{thm}

Subsequently, we adapt this proof to treat in a unified way a broader class of
spatially inhomogeneous problems of the form (\ref{eqinhom}).

\subsection{Method of proof and outline of the paper}
Our proof is motivated from the constraint variational set up of
\cite{alikakosIndiana} but, instead of using energy decreasing local
replacement arguments as a substitute of the maximum principle, we
will use energy controlling local replacements together with a
clearing-out argument. In particular, we do not need to employ the
polar representation that was used in \cite{alikakosIndiana} (see
also the introduction in \cite{bethuel}), that is to write a
function $u\in W^{1,2}(\mathbb{R},\mathbb{R}^n)$ as
\[
u(x)=a_\pm+\rho_\pm(x)\Theta_\pm(x)\ \ \textrm{whenever}\
\rho_\pm(x)=\left|u(x)-a_\pm\right|\neq 0;\ \ u(x)=a_\pm\ \
\textrm{otherwise},
\]
 which turns out to be a rather cumbersome issue (see, however, the observation  below (\ref{eqpolar}) herein),  especially
in the case of the corresponding elliptic problems  (see
\cite{alikakosPreprint}).

In our opinion, besides of rendering the most general
result, our proof is the simplest available.

The outline of the paper is the following: In Section \ref{secprof}
we present the proof of Theorem \ref{thmThm}, and in Section
\ref{secinhomo} we consider  some other extensions to  the inhomogeneous case.
\section{Proof of the main result}\label{secprof}
\begin{proof}[Proof of Theorem \ref{thmThm}]
The  main part of the proof will be devoted in  showing that  there
exists a solution $u \in C^2(\mathbb{R}, \mathbb{R}^n)$ to the
equation
\begin{equation}\label{eqeuler}
u_{xx}=h(x)\nabla W(u),
\end{equation}
and an $L>0$, such that
\begin{equation}\label{eqstrict}
|u(x)-a_-|<\delta,\ \ x\leq -L;\ \ |u(x)-a_+|<\delta,\ \ x\geq L,
\end{equation}
for some small \begin{equation}
\label{eqdeltaD}\delta <\frac{|a_+ -a_-|}{2}.
\end{equation} To this end, as in
\cite{alikakosIndiana}, for $L>2$, let
\begin{equation}\label{eqX1}
X_{L}^-=\left\{u\in W^{1,2}_{loc}(\mathbb{R},\mathbb{R}^n)\ :\
|u(x)-a_-|\leq \delta,\ \ x\leq -L
 \right\},
\end{equation}
\begin{equation}\label{eqX2}
X_{L}^+=\left\{u\in W^{1,2}_{loc}(\mathbb{R},\mathbb{R}^n)\ :\
|u(x)-a_+|\leq \delta,\ \ x\geq +L
 \right\}.
\end{equation}
It is standard to show that there exists a $u_{L}\in X_{L}^- \cap
X_{L}^+$ such that
\begin{equation}\label{eqfinit}
J(u_{L})=\inf_{u\in X_{L}^-\cap X_{L}^+}J(u)<\infty,
\end{equation}
where $J:W^{1,2}_{loc}(\mathbb{R},\mathbb{R}^n) \to [0,\infty]$ is
the associated energy functional
\begin{equation}\label{eqJinho}
J(u)=\int_{\mathbb{R}}^{}\left\{\frac{1}{2}|u_x|^2+h(x)W(u) \right\}dx.
\end{equation}
This was shown in \cite{alikakosIndiana} in the case where $h \equiv 1$, but the general case where $h\geq 0$ can be treated completely analogously (no other property of $h$ besides  reasonable regularity is needed at this point). Our goal is to show that  there exists
$L\gg 1$ such that $u_{L}$ (or some translation of it) satisfies
(\ref{eqstrict}), since this will imply that $u_L$ is a classical
solution to (\ref{eqeuler}). We note that, a-priori, the minimizer
$u_{L}$ is $C^2$ and satisfies the Euler-Lagrange equation
(\ref{eqeuler}) \emph{only in $(-L,L)$ and wherever it is away from
the cylindrical boundary of the constraints}.

By constructing a piecewise linear competitor that is identically
equal to $a_-$ for $x\leq -1$ and equal to $a_+$ for $x\geq 1$, it
is easy to show that
\begin{equation}\label{eqC1}
J(u_{L})\leq C_1,
\end{equation}
where the constant \emph{$C_1>0$ is independent of $L>2$} (an
analogous argument also appears in \cite{cafa-cordoba} and many other papers).

We claim that, given any $d\in (0,\delta]$,  there exists
$\varepsilon \in (0,\frac{d}{2})$, \emph{independent of $L>2$}, such
that
\begin{equation}\label{eqclaim}
\textrm{if}\ \ x_2-x_1\geq 3\ \ \textrm{and} \ \
|u_{L}(x_i)-a_\pm|\leq \varepsilon,\ i=1,2,
\end{equation}
then
\begin{equation}\label{eqclaim2}
|u_{L}(x)-a_\pm|<d,\ \ x\in [x_1,x_2].
\end{equation}
In passing, we note that an analogous property was established independently in \cite[Prop. 8.1]{valdin} for the corresponding problem with nonlocal diffusion, where it is appropriately called 'stickiness property' (see also \cite{smyrnelia} for a more involved application of this property in the local setting).
 It is clear that we only have to verify this claim for the $+$
case. To this end, suppose that $x_1,x_2 \in \mathbb{R}$ and $\varepsilon \in (0,\frac{d}{2})$ are such that the corresponding case of  (\ref{eqclaim}) holds for $u_L$. The main observation is that the minimality
property  of $u_{L}$ implies that there exists a constant $C_2>0$,
\emph{independent of $\varepsilon,x_1,x_2,L$}, such that
\begin{equation}\label{eqenergyBound}
\int_{x_1}^{x_2}\left\{\frac{1}{2}\left|(u_{L})_x\right|^2+h(x)W(u_{L})
\right\}dx\leq C_2\varepsilon.
\end{equation}
In fact, if $W$ was $C^2$ near $a_+$, we would have $\varepsilon^2$ instead of $\varepsilon$ in the above relation.  This follows by comparing the energy of $u_L$ (keep in mind (\ref{eqfinit})) to that of a  $\tilde{u}_L\in X_L^{-}\cap X_L^{+}$ which agrees with $u_{L}$
outside of $(x_1,x_2)$, is identically equal to $a_+$ over $[x_1+1,x_2-1]$,
and is linear in the intermediate interpolation intervals. We point out that the assumption that the distance between $x_1$
and $x_2$ is larger than some universal constant (the number $3$ here is chosen for convenience purposes only) plays a crucial
role in controlling the gradient of $\tilde{u}_L$ in the interpolation zones. One could say
that the competitor $\tilde{u}_L$ is obtained by performing \emph{surgery} on $u_L$; we refer the interested reader to \cite[Rem. 2.3]{ambrosioCabre}
and \cite{chenFH} for related surgery type constructions.
 The desired claim now
follows by applying the clearing-out lemma in \cite{bethuel} (see
Lemma 1 therein). For the sake of completeness, and for future purposes, let us present a
different argument. Suppose to the contrary that there exists $x_*
\in (x_1,x_2)$ such that
\begin{equation}\label{eqpossi}
|u_{L}(x)-a_+|<d,\ \ x\in [x_1,x_*), \ \ \textrm{and}\ \
|u_{L}(x_*)-a_+|=d.
\end{equation}
Note that there exists  a $V\in C[0,\delta]$, $V>0$ on $(0,\delta]$,
such that
\begin{equation}\label{eqlower}
W(a_\pm+\rho \nu)\geq V(\rho)\ \ \forall\ \rho \in [0,\delta],\
\nu\in \mathbb{S}^{n-1}.
\end{equation}
Indeed, plainly set $V(\rho)=\min\{V_-(\rho),V_+(\rho) \}$, where
\begin{equation}\label{eqpolar}
V_\pm(\rho)=\min_{\nu\in {\mathbb{S}}^{n-1}}W(a_\pm+\rho \nu),\ \
\rho\in [0,\delta].
\end{equation}
In passing, we observe that $u_L(x)\neq a_+$, $x\in [x_1,x_*]$ (if not
and $u_L(\bar{x})=a_+$ for some $\bar{x}$,   the function which coincides with $u_L$ for $x<\bar{x}$ and is identically equal to $a_+$ for $x\geq \bar{x}$
would belong in $X_L^{-}\cap X_L^{+}$ while having less energy than the minimizer $u_L$). Armed with this
information, we have
\[\begin{array}{rcl}
    \int_{x_1}^{x_*}\left\{\frac{1}{2}\left|(u_{L})_x\right|^2+h(x)W(u_{L})
\right\}dx & \stackrel{(\ref{eqh0G}),(\ref{eqlower})}{\geq} &
\int_{x_1}^{x_*}\left\{\frac{1}{2}\left|(u_{L}-a_+)_x\right|^2+h_0V\left(|u_{L}-a_+|\right)
\right\}dx \\
      &   &   \\
    \textrm{via the diamagnetic inequality \cite{FH}:} &
    \geq &\int_{x_1}^{x_*}\left\{\frac{1}{2}\left|u_{L}-a_+\right|_x^2+h_0V\left(|u_{L}-a_+|\right)
\right\}dx    \\
      &   &   \\
    \textrm{by Young's inequality:} & \geq &  \sqrt{2h_0}
\int_{x_1}^{x_*}|u_{L}-a_+|_xV^\frac{1}{2}\left(|u_{L}-a_+|\right)
dx
 \\
      &   &   \\
\textrm{by the area formula \cite{leoniBook}:}&=&\sqrt{2h_0}
\int_{\mathbb{R}}^{}V^\frac{1}{2}(\rho)\textrm{card}|u_{L}-a_+|^{-1} \left(\{\rho \}\right)d\rho \\
& &\\
\textrm{from}\ (\ref{eqclaim}), (\ref{eqpossi}):     &  \geq  &\sqrt{2h_0}
      \int_{\frac{d}{2}}^{d}V^\frac{1}{2}(\rho)d\rho,
   \end{array}
\]
where \emph{card} stands for the cardinality and \[|u_{L}-a_+|^{-1} \left(\{\rho \}\right)=\left\{x\in (x_1,x_*)\ :\
\left|u_{L}(x)-a_+\right|=\rho \right\}.\]
It is worthwhile to mention that the use of Young's inequality in related contexts seems to be due to \cite{modica} and is frequently referred to as Modica's trick (recall also the discussion in Subsection \ref{subsecMotiv}).
Therefore, on account of (\ref{eqenergyBound}), we can exclude the
possibility (\ref{eqpossi}) by choosing
\begin{equation}\label{eqepsilon}\varepsilon\in \left(0,\frac{d}{2}\right)\ \ \textrm{such\ that} \ \ \varepsilon<\frac{\sqrt{h_0}}{C_2}
\int_{\frac{d}{2}}^{d}V^\frac{1}{2}(\rho)d\rho,
\end{equation}
which proves the claim.
 It is worth mentioning that, if (\ref{eqAF}) holds, the maximum principle of \cite{alikakosIndiana,alikakosPreprint}, which follows from a sophisticated surgery type argument, asserts the following: if $|u_L(y_i)-a_+|<d$, $i=1,2$, for some $y_1<y_2$ and $d\in (0,\delta]$, then $|u_L(x)-a_+|<d$, $x\in (y_1,y_2)$. In fact, the main difference of our proof with that in \cite{alikakosIndiana} lies in that
 we use the 'asymptotic maximum principle' in (\ref{eqclaim})-(\ref{eqclaim2}) instead of the aforementioned maximum principle
 that was developed and used therein (see also \cite{alikKATZ,alikakosPreprint} for various extensions).

Next, we claim that, for any $\zeta>0$ sufficiently small, there
exists \begin{equation}\label{eqM3}M>3,\end{equation} \emph{independent of $L$}, and a sequence of positive
numbers $x_1^+<x_2^+<\cdots$, with
\begin{equation}\label{eqMineq}
x_1^+\in (0,M),\ \
M<x_{i+1}^+-x_i^+<3M,\ \  i\geq1,\end{equation} such that
\begin{equation}\label{eq1}W\left(u_{L}(x_i^+) \right)\leq \zeta, \
\ i\geq 1.\end{equation} To see this, plainly take
\begin{equation}\label{eqM}
M\geq C_1h_0^{-1}\zeta^{-1},
\end{equation}
where $C_1$ is as in (\ref{eqC1}) (we may assume that $M>3$), and
apply the integral mean value theorem in the intervals $[0,M],
[2M,3M],\cdots$. Analogously, given $\zeta>0$ sufficiently small, we
can find negative numbers $\cdots<x_2^-<x_1^-$, with $x_1^-\in (-M,0)$,
$M<x_i^--x_{i+1}^-<3M$ (increasing the value of $M$ if needed), such
that $W\left(u_{L}(x_i^-) \right)\leq \zeta$, $i\geq 1$.

We also claim that there exists a constant $C_3>0$, independent of $L>2$, such that
\begin{equation}\label{eqC3gp}\left|u_L(x) \right|\leq C_3,\ \  x\in \mathbb{R}.\end{equation} Indeed, in view of the definition (\ref{eqomega}) and the property (\ref{eqC1}),
we can easily adapt the previous argument, leading to (\ref{eqepsilon}), to get that
\begin{equation}\label{eqtriviali}
C_1\geq \sqrt{2h_0} \int_{\min_{x\in \mathbb{R}}|u_L|}^{\max_{x\in \mathbb{R}}|u_L|}\sqrt{\omega(s)}ds,
\end{equation}
(see also \cite{leonESaiM}). Now, the desired estimate follows at once from (\ref{eqLeoni}) and the trivial observation that
$\min_{x\in \mathbb{R}}|u_L|\leq C_4$ for some constant $C_4>0$ that is independent of $L>2$ (just take $C_4=|a_-|+\delta$).

Let $\varepsilon>0$ be as in (\ref{eqepsilon}) with \begin{equation}\label{eqdeltaLi}d=\delta,\end{equation} so that property (\ref{eqclaim})-(\ref{eqclaim2}) is valid. Then, let
$\zeta>0$ be such that the following property holds: \begin{equation}\label{eq2}W(u)\leq \zeta \ \ \textrm{and}\ \ |u|\leq C_3 \ \
\textrm{imply\ that}\ \ |u-a_-|\leq \varepsilon \ \textrm{or}\
|u-a_+|\leq \varepsilon,\end{equation}
which of course is possible thanks to (\ref{eqass1}).
  We then choose
\[L=1000M,\] where $M>3$ is any fixed number satisfying (\ref{eqM}). We note that this choice of $L$ will turn out to be much larger than what is actually needed, that is we make it for convenience purposes only.  From
(\ref{eqclaim}), (\ref{eqclaim2}), (\ref{eqMineq}), (\ref{eq1}), and
(\ref{eq2}), it follows readily that
\begin{equation}\label{eqknow}
|u_{L}(x)-a_-|<\delta\ \ \textrm{if}\ \ x\leq -1010M; \ \
|u_{L}(x)-a_+|<\delta\ \ \textrm{if}\ \ x\geq 1010M.
\end{equation}
Indeed, we first note that (\ref{eqMineq}) certainly implies that there exists $i_0\in \mathbb{N}$
such that
\begin{equation}\label{eqMX}
x_{i_0}^+\in (1000M,1010M).
\end{equation}
Then, in view of  (\ref{eqdeltaD}) and (\ref{eqX2}),
we obtain from (\ref{eq1}) with $i\geq i_0$, (\ref{eqC3gp}) and property (\ref{eq2}) that
\[
|u_L(x_i^+)-a_+|\leq \varepsilon,\ \ i\geq i_0.
\]
So, thanks to (\ref{eqM3}), (\ref{eqMineq}) and (\ref{eq1}), we can use the property (\ref{eqclaim})-(\ref{eqclaim2}) in each interval
$(x_i^+,x_{i+1}^+)$, $i\geq i_0$, to deduce that
\[
|u_L(x)-a_+|<\delta,\ \ x\in [x_{i_0}^+,\infty)
\]
(keep in mind (\ref{eqdeltaLi})). The second relation in (\ref{eqknow}) now follows at once if we recall (\ref{eqMX}).
Similarly we can show the validity of the first relation.

In the remainder of the proof, we will further restrict $M$ to be an integer multiple of the period $T$ of $h$.

In view of (\ref{eq1}) and (\ref{eq2}), only two possibilities can
occur:

\textbf{(1)} $|u_{L}(x_{1}^+)-a_+|\leq \varepsilon$. Then, by the
property (\ref{eqclaim})--(\ref{eqclaim2}), the condition (\ref{eqdeltaLi}), and the second part of
(\ref{eqknow}), we infer that \begin{equation}\label{eqgives} |u_{L}(x)-a_+|<\delta \ \ \textrm{for}\ \ x\geq
x_{1}^+\in (0,M).\end{equation}
In light of the first relation in (\ref{eqknow}) and the above estimate, if $u_{L}$ does not touch the cylindrical constraint somewhere on $[-1010M,-1000M]$  we are done. In any case, its translate
\[
v_L(\cdot)=u_{L}(\cdot-20M)
\]
does satisfy the desired relation (\ref{eqstrict}).
Indeed,   the first relation in (\ref{eqknow}) yields that
\[
|v_L(x)-a_-|<\delta \ \ \textrm{if}\ \ x\leq -990M=-L+10M.
\]
On the other side, relation (\ref{eqgives}) gives that
\[
|v_L(x)-a_+|<\delta,\ \ x\geq 21M=L-979M.
\]
Moreover, since $h$ is $T$-periodic and $M$ is an integer multiple of $T$, we have that $v_L$ has the same energy as $u_L$. We note that the periodicity of $h$ is used only at this point.
In other words, $v_L$ is also a minimizer of problem (\ref{eqfinit}) which further satisfies (\ref{eqstrict}) (i.e. it does not touch the cylindrical  constraints).


\textbf{(2)} $|u_{L}(x_{1}^+)-a_-|\leq \varepsilon$. Then,  we
  have that $|u_{L}(x)-a_-|< \delta$ for $x\leq x_{1}^+\in (0,M)$
(from (\ref{eqclaim})-(\ref{eqclaim2}) and the first part of
(\ref{eqknow})). In that case, as before,
 replacing $u_{L}$ by the translated minimizer
 $u_{L}(\cdot+20M)$, if necessary, we find that (\ref{eqstrict})
 holds, as desired.

The above argument of using a suitable translate of $u_L$ is mainly motivated from \cite{alikakosIndiana}.
Actually, under the monotonicity assumption (\ref{eqAF}), it was shown in the latter reference that $u_L$ can touch the cylindrical constraints  at most at one of
the points $x=-L$ or $x=L$, provided that $L$ is sufficiently large. So, in contrast to the general case at hand, they could use any sufficiently small translation in the appropriate direction.

We have thus shown that the minimizer $u_{L}$ satisfies
(\ref{eqstrict}). In particular, by standard arguments (see
\cite{alikakosIndiana}), it induces  a classical solution to
(\ref{eqeuler}). To complete the proof of the theorem, we will show
that \begin{equation}\label{eqLimes}\lim_{x\to \pm \infty}u_L(x)=a_\pm.\end{equation}
 Indeed, for any arbitrarily small   $\tilde{d}>0$,
 let $\tilde{\varepsilon}\in (0,\frac{\tilde{d}}{2})$ be such that the corresponding property to (\ref{eqclaim})-(\ref{eqclaim2}) holds.
As before, using (\ref{eqC1}) and (\ref{eq2}), we can find a sequence $\{ \tilde{x}_i^+\}$
such that $\tilde{x}_{i+1}^+- \tilde{x}_i^+\geq 3$, $\tilde{x}_i^+\geq L$, $i\geq 1$, satisfying
$|u_L(\tilde{x}_i^+)-a_+|\leq \tilde{\varepsilon}$ for $i\geq 1$. Thus, by the aforementioned property (\ref{eqclaim})-(\ref{eqclaim2}), we deduce that
$|u_L(x)-a_+|< \tilde{d}$, $x\geq \tilde{x}_1^+$, which clearly implies the validity of the $+$ case in (\ref{eqLimes}).
Similarly we can show the $-$ case.

The proof of the theorem is complete.
\end{proof}

\begin{rem}\label{rem1}
\emph{In the scalar case ($n=1$), further assuming that $a_\pm$ are
non-degenerate minima of $W$, this problem was considered in
\cite{alessioPeriod}, and for $W$ as above in \cite{bonhoure}.}
\emph{It is easy to see that, when $n=1$, the above theorem as well as Theorem \ref{thminhomconstnt} below do not need the assumption (\ref{eqLeoni}).}
\emph{In fact, by appropriately modifying $W$ in the two intervals outside of its global minima (see for example \cite[Ch. 1]{amster}), we can capture a heteroclinic solution to the resulting system with values strictly between them. Of course this is also a heteroclinic to the original problem.}
\end{rem}

\begin{rem}
\emph{For a recent application of the above theorem, we refer to \cite{sourdisAML}.}
\end{rem}
\begin{rem}
	\emph{The proof of Theorem \ref{thmThm} carries over without difficulty to
		the quasi-linear setting:
		\[
		\left(|u_x|^{p-2}u_x \right)_x=\nabla W(u),\ \ \lim_{x\to \pm
			\infty}u(x)=a_\pm,\ \ (p>2),
		\]
		at least when (\ref{eqinf}) is assumed.
		This problem was considered  in \cite{kar}, and the references therein, under
		assumption
		(\ref{eqAF}). The only essential difference is that one has to
		modify slightly the proof of the clearing-out lemma of \cite{bethuel} by
		using the H\"{o}lder inequality instead of the  Cauchy-Schwarz .
	}\end{rem}

\begin{rem}\label{remAlik}
\emph{Recently in \cite{alikakosPreprint}, building on the arguments
of \cite{alikakosIndiana} which rely on assumption (\ref{eqAF}), the
authors constructed heteroclinic connections for semilinear elliptic
systems of the form $\Delta u=\nabla W(u)$ in singly periodic
domains of $\mathbb{R}^m$ with Neumann boundary conditions. In fact, as is pointed
out, their approach can be extended to construct heteroclinic
connections for the problem
\begin{equation}\label{eqalikh}\left\{\begin{array}{c}
                                         \Delta u=h(x_1,\cdots,x_m)\nabla W(u) \\
                                         \textrm{in\ cylindrical\ domains\ (in\ the}  \ x_1 \
\textrm{direction)}\\
\textrm{with\ Neumann\ boundary\ conditions},
                                       \end{array}\right.
\end{equation} where $h$ is positive and periodic in $x_1$. It would
be interesting to see how much our approach can be pushed towards
this direction.}
\end{rem}

\section{Further inhomogeneous problems}\label{secinhomo}

\subsection{The asymptotically constant inhomogeneity}
\begin{thm}\label{thminhomconstnt}
\emph{Assume that $h\in C(\mathbb{R})$  satisfies (\ref{eqh0G}),
\begin{equation}\label{eqh}\lim_{ x\to \pm \infty}h(x)=h_\infty\in (0,\infty)\ \
\textrm{and}\ \ h(x)\leq h_\infty,\ x\in \mathbb{R}.\end{equation}
Under assumptions (\ref{eqass1}) and (\ref{eqLeoni}) on $W$, there exists a
solution to the problem (\ref{eqinhom}).
}\end{thm}
\begin{proof}
The main difference of the problem at hand with the previous ones is
that there is no translation invariance (continuous or discrete).

As before, for $L>2$, let
\begin{equation}\label{eqJ-}
m_L=\inf_{u\in X_L^-\cap X_L^+} J(u),
\end{equation}
where $X_L^\pm$ are as in (\ref{eqX1})--(\ref{eqX2}), and the energy functional $J$ is as in
(\ref{eqJinho}).
As we mentioned in the proof of Theorem \ref{thmThm},
it is easy
to  show that the infimum is attained at some $u_L\in X_L^-\cap
X_L^+$.

Motivated from \cite{bartsch}, where ground states to the nonlinear
Schr\"{o}dinger equation with potential $h$ were considered, we will
compare $m_L$ with the 'limiting energy'
\[
m_{\infty,L}=\inf_{u\in X_L^-\cap
X_L^+}\int_{{\mathbb{R}}}^{}\left\{\frac{1}{2}|u_x|^2+h_\infty
W(u)\right\}dx.
\]
As we have already shown in Theorem \ref{thmThm}, the above infimum is attained
by a  classical solution $u_{\infty,L}\in X_L^-\cap X_L^+$ of the
problem
\[
u_{xx}=h_\infty\nabla W(u),\ \ \lim_{x\to \pm \infty} u(x)=a_\pm,
\]
provided that $L$ is sufficiently large. Clearly, since $u_{\infty,L}$ is continuous and (\ref{eqdeltaD}) holds, there exists $x_L\in \mathbb{R}$
such that
\begin{equation}\label{eqexplan}
\left|u_{\infty,L}(x_L)-a_-\right|\geq \delta \ \ \textrm{and}\ \
\left|u_{\infty,L}(x_L)-a_+\right|\geq \delta.
\end{equation}

Observe that all the properties in the proof of Theorem \ref{thmThm}
up to (\ref{eq2}) remain true for this $u_L$ as well as for $u_{\infty,L}$ (recall also a related comment in Case (1) therein), with the same constants in fact as those in the aforementioned theorem. In light of this, let us keep the same
notation.

 We may assume that
$h(x)<h_\infty$ somewhere, say that
\begin{equation}\label{eq-}h(x)<h_\infty,\ \ x\in (x_-,x_+),\end{equation} for some
$x_-,x_+\in \mathbb{R}$. Thanks to (\ref{eqexplan}), by translating $u_{\infty,L}$ if
necessary, we may assume that
\begin{equation}\label{eqexplan22}
\left|u_{\infty,L}(x_-)-a_-\right|\geq \delta \ \ \textrm{and}\ \
\left|u_{\infty,L}(x_-)-a_+\right|\geq \delta.
\end{equation}
Abusing notation, we will keep denoting by $u_{\infty,L}$ the possibly translated solution.
The main point is that, by increasing $L$ if needed,  the possibly new  $u_{\infty,L}$ is still in $X_L^-\cap
X_L^+$. Indeed, as in the proof of Theorem \ref{thmThm}, there exist $z_{1,L}\in (x_--M,x_-)$, $z_{2,L}\in (x_-,x_-+M)$
\[
\left|u_{\infty,L}(z_{i,L})-a_-\right|\leq \varepsilon \ \ \textrm{or}\ \
\left|u_{\infty,L}(z_{i,L})-a_+\right|\leq \varepsilon,\ \ i=1,2.
\]
Hence, taking into consideration  property (\ref{eqclaim})-(\ref{eqclaim2}) and (\ref{eqexplan22}), we infer that
\[
\left|u_{\infty,L}(x)-a_-\right|< \delta, \ \ x\leq z_{1,L};\ \
\left|u_{\infty,L}(x)-a_+\right|< \delta,\ \ x\geq z_{2,L},
\]
which clearly implies that $u_{\infty,L} \in X_L^-\cap
X_L^+$ provided that $L\geq |x_-|+M$.

 By (\ref{eqexplan22}) and the analog of
(\ref{eqC1}), it is easy to see that
\begin{equation}\label{eq--}
\left|u_{\infty,L}(x)-a_-\right|\geq \frac{\delta}{2} \ \
\textrm{and}\ \ \left|u_{\infty,L}(x)-a_+\right|\geq
\frac{\delta}{2}\ \ \textrm{for}\ \ x\in
\left[x_-,x_-+\frac{\delta^2}{8C_1} \right],
\end{equation}
where $C_1$ is as in (\ref{eqC1}),
(the point being that this interval is independent of large $L$).
Indeed, if $x\in
 \left[x_-,x_-+\frac{\delta^2}{8C_1} \right]$, letting
 $\rho_\pm(x)=|u_{\infty,L}(x)-a_\pm|$, we have
 \begin{equation}\label{eqremai}
\left|\rho_\pm(x)-\rho_\pm(x_-)\right|\leq
\int_{x_-}^{x}|u_{\infty,L}-a_\pm|_tdt \leq
\int_{x_-}^{x}\left|(u_{\infty,L})_t\right|dt\leq
|x-x_-|^\frac{1}{2} (2C_1)^\frac{1}{2}\leq \frac{\delta}{2}.
 \end{equation}
 Then, using $u_{\infty,L}$ as a
test function, we find that
\begin{equation}\label{eqcontra}
  \begin{array}{lll}
    m_L & \leq & \int_{{\mathbb{R}}}^{}\left\{\frac{1}{2}|(u_{\infty,L})_x|^2+h(x)W(u_{\infty,L})\right\}dx \\
      &   &   \\
      &  = & \int_{\mathbb{R}}^{}\left\{\frac{1}{2}|(u_{\infty,L})_x|^2+h_\infty W(u_{\infty,L})\right\}dx
      +\int_{\mathbb{R}}^{}\left(h(x)-h_\infty\right)W(u_{\infty,L})dx \\
      &   &   \\
 \textrm{via}\ (\ref{eq-}),(\ref{eq--})     & \leq & m_{\infty,L}-c
  \end{array}
\end{equation}
where $c>0$ is independent of large $L$.

 This time we let
\[
L=L_j=jM,
\]
with $j$ a sufficiently large integer that is to be determined so
that (\ref{eqstrict}) holds, which in particular will imply that
$u_L$ is a classical solution to
\begin{equation}\label{eqinhomSb}u_{xx}=h(x)\nabla W(u).\end{equation} Suppose, to the
contrary, that there exists a sequence of $L_j\to \infty$ such that
(\ref{eqstrict}) with $L=L_j$ is violated at some $x_j\leq -L_j$ (the
other case is completely analogous). Then, similarly to (\ref{eqknow}), denoting $u_{L_j}$ by $u_j$, we
would have that
\begin{equation}\label{eqknowInhom}
|u_{j}(x)-a_-|<\delta\ \ \textrm{if}\ \ x\leq -(j+10)M; \ \
|u_{j}(x)-a_+|<\delta\ \ \textrm{if}\ \ x\geq (j+10)M.
\end{equation}
In particular, this implies that $x_j \in \left(-(j+10)M,-jM\right]$. Thus, as in the proof of Theorem \ref{thmThm}, this gives that
\begin{equation}\label{eqmidn}
|u_{j}(x)-a_+|<\delta\ \ \textrm{if}\ \ x\geq -(j-1)M.
\end{equation}
From the  above relation (which implies that
$u_j$ solves (\ref{eqinhomSb}) for $x\geq -(j-1)M$), making use of
Arczela-Ascoli's theorem and the standard diagonal argument, passing
to a subsequence if needed, we find that
\begin{equation}\label{eqUloc1}
u_{j}\to {U}\ \ \textrm{in}\ C_{loc}({\mathbb{R}},{\mathbb{R}}^n),
\end{equation}
where ${U}$ satisfies
\begin{equation}\label{eqU}
{U}_{xx}-h(x)\nabla W({U})=0,\ \ |{U}(x)-a_+|\leq \delta, \ \ x\in
{\mathbb{R}}.
\end{equation}
 Moreover, from the
minimality of $u_j$, and the second part of (\ref{eqknowInhom}), it
follows readily that $U$ is a minimizer of the energy subject to its
boundary conditions, that is
\[
J(U)\leq J(U+\varphi)\ \ \ \forall\ \varphi\in W^{1,2}_0(I,{\mathbb{R}}^n) \ \
\textrm{and\ any\ bounded\ interval}\ I\subset \mathbb{R},
\]
(this can be proven as in \cite{danceryanCVPDE}).
Similarly to (\ref{eqLimes}), we find that
\[
\lim_{x\to \pm \infty}U(x)=a_+.
\]
Then, by means of the analog of property (\ref{eqclaim})-(\ref{eqclaim2}), we get that \begin{equation}\label{eqUloc2}U\equiv a_+.\end{equation}
Actually, in
the case where (\ref{eqAF}) holds, not necessarily with   strict
monotonicity, this can also be deduced by the weak sub-harmonicity of
the function $|U-a_+|$, which follows directly from
(\ref{eqU}). It follows from (\ref{eqUloc1}) and (\ref{eqUloc2}) that
\begin{equation}\label{eqget}
W(u_j)\to 0\ \ \textrm{in}\ C_{loc}(\mathbb{R}).
\end{equation}
 On
the other hand, we have
\[
\begin{array}{rcl}
  m_{L_j} & = & \int_{{\mathbb{R}}}^{}\left\{\frac{1}{2}|(u_j)_x|^2+h(x)W(u_j)\right\}dx \\
    &   &   \\
    &  = & \int_{\mathbb{R}}^{}\left\{\frac{1}{2}|(u_j)_x|^2+h_\infty W(u_j)\right\}dx+
    \int_{\mathbb{R}}^{}\left(h(x)-h_\infty\right) W(u_j)dx  \\
    &   &   \\
u_j \in X_{L_j}^-\cap
X_{L_j}^+\ :&  \geq & m_{\infty,L_j}+
    \int_{\mathbb{R}}^{}\left(h(x)-h_\infty\right) W(u_j)dx  \\
    &   &   \\
\textrm{via}\ (\ref{eqC1}),\ (\ref{eqh}),\  (\ref{eqget}):    & \geq  &
m_{\infty,L_j}+o(1),
\end{array}
\]
where $o(1)\to 0$ as $j\to \infty$, which contradicts
(\ref{eqcontra}). In more detail,  to get the last relation, we estimate as follows:
\[\begin{array}{rcl}
    \int_{\mathbb{R}}^{}\left|\left(h(x)-h_\infty\right) W(u_j)\right|dx & = & \int_{|x|<K}^{}\left|\left(h(x)-h_\infty\right) W(u_j)\right|dx+\int_{|x|>K}^{}\left|\left(h(x)-h_\infty\right) W(u_j)\right|dx \\
      &   &   \\
      & \leq & 4Kh_\infty \max_{|x|\leq K }W(u_j)+C_1 h_0^{-1}\sup_{|x|\geq K}\left|h(x)-h_\infty\right|,
  \end{array}
\]
for any $K>0$. Then, given any $\epsilon>0$,  we choose $K$ so that the second term is smaller than $\epsilon/2$ and subsequently  $j_0$ so that the first term is smaller than $\epsilon/2$ for $j\geq j_0$.

Having established that (\ref{eqstrict}) holds for sufficiently
large $L$, the rest of the proof proceeds verbatim as that of
Theorem \ref{thmThm}.
\end{proof}
\begin{rem}
	\emph{We note that the first condition in (\ref{eqh}) was used only at the very end of the above proof for showing that
	\[\int_{\mathbb{R}}^{}\left(h(x)-h_\infty\right) W(u_j)dx\to 0\ \ \textrm{as}\ \ j\to \infty.\]}
\emph{Interestingly enough, in view of (\ref{eqC3gp}), the above relation can also be deduced from (\ref{eqget}) and Lebesgue's dominated convergence theorem if we  assume instead that
\[h-h_\infty \in L^1(\mathbb{R}).\]}
\end{rem}

\begin{rem}\emph{
Condition (\ref{eqh0G}) was used crucially in obtaining  the uniform estimate (\ref{eqC3gp}) (recall (\ref{eqtriviali})), which allowed us to get (\ref{eqUloc1}). Nevertheless, with some care, we can still obtain a uniform estimate (with respect to $L$) under the weaker condition:
\begin{equation}\label{eqmem}
h(x)\geq h_0,\ \ |x|\geq N \ \ \textrm{for\ some}\ h_0,N>0;\ \ h(x)\geq0,\ \ x\in \mathbb{R},
\end{equation}
instead of (\ref{eqh0G}). Indeed, the proof of (\ref{eqC3gp}) gives us first a uniform estimate for $|u_L|$ on $|x|\geq N$. Then, this can be extended in the remaining region by arguing as in (\ref{eqremai}). Clearly, the validity of (\ref{eqknowInhom}) is not affected by weakening the assumption (\ref{eqh0G}) to (\ref{eqmem}). The same is also true for (\ref{eqmidn}). Indeed, as in Theorem \ref{thmThm} (recall especially (\ref{eq1})), there exist $\xi_{1,j}\in \left(-jM,-(j-1)M\right)$, $\xi_{2,j}\in \left(-(N+M),-N\right)$ such that $\left|u_j(\xi_{i,j})-a_+ \right|\leq \varepsilon$, $i=1,2,$ $j\geq 1$. Now, property (\ref{eqclaim})-(\ref{eqclaim2}) gives us the desired relation in $\left(-(j-1)M,\xi_{1,j}\right)$. Moreover, by relation (\ref{eqremai}) with   fixed end point $\xi_{1,j}$, reducing $\varepsilon>0$ if needed, we deduce that the same holds in $(\xi_{1,j},\xi_{2,j})$. Lastly, in the remaining interval $(\xi_{2,j},\infty)$ we just apply property (\ref{eqclaim})-(\ref{eqclaim2}) between $\xi_{2,j}$ and $(j+11)M$.   We argue similarly for showing (\ref{eqUloc2}) under these conditions. \emph{We therefore conclude that the assertion of Theorem \ref{thminhomconstnt} continues to hold even if (\ref{eqh0G}) is replaced by the weaker condition (\ref{eqmem}).}}

\emph{In the case where $h$ is periodic, we can repeat the above procedure in a periodic fashion, and find that the assertion of Theorem \ref{thmThm} still holds even if (\ref{eqh0G}) is weakened to $h\geq 0$, $h$ nontrivial.}
\end{rem}

\begin{rem}\label{rem2}
\emph{Using a different variational argument, Theorem \ref{thminhomconstnt} was proven
in the scalar case in \cite[Thm. 2.2]{bonhoure} (see also \cite{sprandlin} for a result which allows $h_\infty - h(x)$ to change sign for
arbitrarily large values of $|x|$).
}
\end{rem}
\begin{rem}
\emph{It may be plausible that the above theorem generalizes to the case where $h_\infty$ is a periodic function (with the obvious interpretation of  (\ref{eqh})). A related result for the scalar problem can be found in \cite{alves}. It is worth mentioning that the results  in \cite{chenTzeng} do not require the corresponding inhomogeneity to be asymptotically periodic
(however they require further assumptions on the corresponding potential which include the  nondegeneracy of its global minima).
}\end{rem}


\begin{rem}
\emph{In \cite{korman}, the authors used the method of upper and lower solutions together with an approximation by large finite intervals to show that there is a unique solution to the problem
\[
u'' =h(x)(u^3-u), \ x\in \mathbb{R};\ \ \lim_{x\to \pm \infty}u(x)=\pm 1,
\]
which is an odd and strictly increasing function, under the  assumptions that
$h \in C^1(\mathbb{R})$ is even, $h'<0$ for almost all $x>0$, and $\lim_{x\to +\infty} h(x)>0$. Observe that   this result  is not contained in our Theorem \ref{thminhomconstnt} because the second assumption in (\ref{eqh}) is violated}.
\emph{Nevertheless, an inspection of the proof of   Theorem \ref{thmThm} (see also Remark \ref{rem1}) yields that there exists an odd solution to the scalar problem
\[
u''=h(x)W'(u), \ x\in \mathbb{R};\ \ \lim_{x\to \pm \infty}u(x)=\pm 1,
\]
such that $u>0$ in $(0,\infty)$,
provided that the following assumptions are fulfilled: $h\in C(\mathbb{R})$, $W\in C^1(\mathbb{R})$ are even, $h\geq 0$, $\liminf_{x\to + \infty} h(x)>0$, and  $W(u)>0$ for $u\in [0,\infty)\setminus \{ 1\}$.  The main point is that we can 'pin' the minimizer $u_L$ at the origin by restricting ourselves to the class of odd functions}.   \end{rem}

\subsection{The diverging inhomogeneity}
\begin{thm}\label{thmDiverg}
\emph{Assume that $h\in C(\mathbb{R})$ is nonnegative, and
\begin{equation}\label{eqh++++}\lim_{ x\to \pm \infty}h(x)=\infty.\end{equation} Under solely the
assumption  (\ref{eqass1}) on $W$, there exists a
solution to the problem (\ref{eqinhom}).
}\end{thm}
\begin{proof}
Our strategy remains the same. We consider the constraint
minimization problem (\ref{eqstrict})-(\ref{eqJinho}) and show that any
minimizer $u_L$ (which exists by standard arguments) satisfies
(\ref{eqstrict}), provided that $L$ is sufficiently large. Clearly,
estimate (\ref{eqC1}) holds (abusing notation).

We claim that, for large $L$, we have that
\[
|u_L(x)-a_+|<\delta,\ \  x\geq L.\] Indeed, suppose to the contrary
that there exists $x_+\geq L$ such that $|u_L(x_+)-a_+|=\delta$ (we have suppressed the obvious dependence of $x_+$ on $L$, $x_+$ is not related to that in the proof of Theorem \ref{thminhomconstnt}).
Then, arguing as we did before for showing (\ref{eq--}), we find that
\[
\delta \geq \left|u_{L}(x)-a_+\right|\geq \frac{\delta}{2} \ \ \textrm{for}\ \
x\in \left[x_+,x_++\frac{\delta^2}{8C_1} \right]. \] In turn, this
implies that
\[
W\left(u_{L}(x) \right)\geq c>0, \ \ x\in
\left[x_+,x_++\frac{\delta^2}{8C_1} \right],
\]
where \emph{the constant $c>0$ is independent of large $L$}. On the
other hand, if $L$ is sufficiently large, the above relation
contradicts the fact that
\[
\int_{L}^{\infty}W\left(u_{L}(t) \right)dt\to 0 \ \ \textrm{as}\ \
L\to \infty,
\]
which follows directly from (\ref{eqC1}) and (\ref{eqh++++}).
Analogously, we can show that \[ |u_L(x)-a_-|<\delta,\ \  x\leq
-L.\]

Having established that $u_L$ satisfies (\ref{eqstrict}) (and as a
consequence (\ref{eqinhomSb})), for sufficiently large $L$,  we can
proceed in a similar manner to show that it also satisfies the
desired asymptotic behavior at respective infinities.
\end{proof}

\begin{rem}\label{rem3}
\emph{If $h(x)>0,\ x\in \mathbb{R}$, the above theorem is contained in
\cite{izJde2007}.
}\end{rem}


\begin{rem}\label{rem4}
\emph{In \cite{sourdisConley}, relying on the oddness of the nonlinearity,
we used a shooting argument to show that there exists a unique odd
solution to the problem
\[
u_{xx}=|x|^\alpha (u^3-u),\ \ \lim_{x\to \pm \infty}u(x)=\pm 1,
\]
where $\alpha>0$. Moreover, this solution is increasing and
asymptotically stable. This heteroclinic connection  describes the
profile of the transition layer, near $x=0$, of the singular
perturbation problem (\ref{eqNaka}) with $h \sim |x|^\alpha$ as
$x\to 0$ and $h>0$ elsewhere (here $W(u)=\frac{(u^2-1)^2}{4}$).
}\end{rem}
\begin{rem}
\emph{The results in this paper generalize straightforwardly to the more general class of systems $u''=\nabla_u W(x,u)$.}
\end{rem}
\section*{Acknowledgments} This project has received funding from the European Union's Seventh Framework programme for  research and innovation under the Marie Sk\l{}odowska-Curie grant agreement No 609402-2020 researchers: Train to Move (T2M).

\end{document}